\documentclass[12pt]{amsart}
\usepackage{amsmath,amssymb,amsbsy,amsfonts,latexsym,amsopn,amstext,
                                               amsxtra,euscript,amscd,bm}
                    
\usepackage{url}
\usepackage[colorlinks,linkcolor=blue,anchorcolor=blue,citecolor=blue]{hyperref}
\usepackage{color}

\usepackage{comment}

\begin{document}
\newtheorem{problem}{Problem}
\newtheorem{theorem}{Theorem}
\newtheorem{lemma}[theorem]{Lemma}
\newtheorem{claim}[theorem]{Claim}
\newtheorem{cor}[theorem]{Corollary}
\newtheorem{prop}[theorem]{Proposition}
\newtheorem{definition}{Definition}
\newtheorem{question}[theorem]{Question}

\def\cA{{\mathcal A}}
\def\cB{{\mathcal B}}
\def\cC{{\mathcal C}}
\def\cD{{\mathcal D}}
\def\cE{{\mathcal E}}
\def\cF{{\mathcal F}}
\def\cG{{\mathcal G}}
\def\cH{{\mathcal H}}
\def\cI{{\mathcal I}}
\def\cJ{{\mathcal J}}
\def\cK{{\mathcal K}}
\def\cL{{\mathcal L}}
\def\cM{{\mathcal M}}
\def\cN{{\mathcal N}}
\def\cO{{\mathcal O}}
\def\cP{{\mathcal P}}
\def\cQ{{\mathcal Q}}
\def\cR{{\mathcal R}}
\def\cS{{\mathcal S}}
\def\cT{{\mathcal T}}
\def\cU{{\mathcal U}}
\def\cV{{\mathcal V}}
\def\cW{{\mathcal W}}
\def\cX{{\mathcal X}}
\def\cY{{\mathcal Y}}
\def\cZ{{\mathcal Z}}

\def\A{{\mathbb A}}
\def\B{{\mathbb B}}
\def\C{{\mathbb C}}
\def\D{{\mathbb D}}
\def\E{{\mathbb E}}
\def\F{{\mathbb F}}
\def\G{{\mathbb G}}
\def\I{{\mathbb I}}
\def\J{{\mathbb J}}
\def\K{{\mathbb K}}
\def\L{{\mathbb L}}
\def\M{{\mathbb M}}
\def\N{{\mathbb N}}
\def\O{{\mathbb O}}
\def\P{{\mathbb P}}
\def\Q{{\mathbb Q}}
\def\R{{\mathbb R}}
\def\S{{\mathbb S}}
\def\T{{\mathbb T}}
\def\U{{\mathbb U}}
\def\V{{\mathbb V}}
\def\W{{\mathbb W}}
\def\X{{\mathbb X}}
\def\Y{{\mathbb Y}}
\def\Z{{\mathbb Z}}

\def\ep{{\mathbf{e}}_p}
\def\em{{\mathbf{e}}_m}
\def\eq{{\mathbf{e}}_q}

\def\scr{\scriptstyle}
\def\\{\cr}
\def\({\left(}
\def\){\right)}
\def\[{\left[}
\def\]{\right]}
\def\<{\langle}
\def\>{\rangle}
\def\fl#1{\left\lfloor#1\right\rfloor}
\def\rf#1{\left\lceil#1\right\rceil}
\def\le{\leqslant}
\def\ge{\geqslant}
\def\eps{\varepsilon}
\def\mand{\qquad\mbox{and}\qquad}

\def\sssum{\mathop{\sum\ \sum\ \sum}}
\def\ssum{\mathop{\sum\, \sum}}
\def\ssumw{\mathop{\sum\qquad \sum}}

\def\vec#1{\mathbf{#1}}
\def\inv#1{\overline{#1}}
\def\num#1{\mathrm{num}(#1)}
\def\dist{\mathrm{dist}}

\def\fA{{\mathfrak A}}
\def\fB{{\mathfrak B}}
\def\fC{{\mathfrak C}}
\def\fU{{\mathfrak U}}
\def\fV{{\mathfrak V}}

\newcommand{\bflambda}{{\boldsymbol{\lambda}}}
\newcommand{\bfxi}{{\boldsymbol{\xi}}}
\newcommand{\bfrho}{{\boldsymbol{\rho}}}
\newcommand{\bfnu}{{\boldsymbol{\nu}}}

\def\GL{\mathrm{GL}}
\def\SL{\mathrm{SL}}

\def\Hba{\overline{\cH}_{a,m}}
\def\Hta{\widetilde{\cH}_{a,m}}
\def\Hb1{\overline{\cH}_{m}}
\def\Ht1{\widetilde{\cH}_{m}}

\def\flp#1{{\left\langle#1\right\rangle}_p}
\def\flm#1{{\left\langle#1\right\rangle}_m}
\def\dmod#1#2{\left\|#1\right\|_{#2}}
\def\dmodq#1{\left\|#1\right\|_q}

\def\Zm{\Z/m\Z}

\def\Err{{\mathbf{E}}}

\newcommand{\comm}[1]{\marginpar{%
\vskip-\baselineskip 
\raggedright\footnotesize
\itshape\hrule\smallskip#1\par\smallskip\hrule}}
\pagestyle{plain}
\def\xxx{\vskip5pt\hrule\vskip5pt}

 \author[M. Bordignon] {Matteo Bordignon}
\address{School of Science, The University of New South Wales Canberra, Australia}
\email{m.bordignon@student.adfa.edu.au}

 \author[B. Kerr] {Bryce Kerr }
\address{School of Science, The University of New South Wales Canberra, Australia}
\email{b.kerr@adfa.edu.au}
\thanks{The second author was supported by Australian Research Council Discovery Project DP160100932.}
\title{\bf An explicit P\'{o}lya-Vinogradov inequality via Partial Gaussian sums}

\date{\today}

\begin{abstract}
In this paper we obtain a new fully explicit constant for the P\'olya-Vinogradov inequality for squarefree modulus. Given a primitive character $\chi$ to squarefree modulus $q$, we prove the following upper bound
\begin{align*}
\left| \sum_{1 \le n\le N} \chi(n) \right|\le c \sqrt{q} \log q,
\end{align*}
where $c=1/(2\pi^2)+o(1)$ for even characters and $c=1/(4\pi)+o(1)$ for odd characters, with an explicit  $o(1)$ term. This improves a result of Frolenkov and Soundararajan for large $q$. We proceed via partial Gaussian sums rather than the usual Montgomery and Vaughan approach of exponential sums with multiplicative coefficients. This allows a power saving on the minor arcs rather than a factor of $\log{q}$ as in previous approaches and is an important factor for fully explicit bounds.
\end{abstract}

\maketitle


\maketitle
\section{Introduction}

Given two integers $N,q$ and a primitive character $\chi$ modulo $q$ consider the sums
\begin{align*}
S(\chi):=\max_{N\le q}\left| \sum_{1 \le n\le N} \chi(n) \right|.
\end{align*}
A bound, proven independently by P\'olya and Vinogradov in the early 1900s, is the following
\begin{align}
\label{eq:SchiN}
S(\chi) \le c \sqrt{q} \log q,
\end{align}
for some absolute constant $c$. For long character sums this inequality has remainded the sharpest known and an important problem is to improve on the $\log{q}$ factor in~\eqref{eq:SchiN}. This problem is more or less resolved assuming the Generalized Riemann Hypothesis.
Paley \cite{Paley} proved that there exist infinitely many integers $q$ and primitive characters $\chi$ modulo $q$, such that 
\begin{align*}
S(\chi) \gg \sqrt{q} \log \log q,
\end{align*}
and it was proven by Montgomery and Vaughan \cite{MV} that $S(\chi)$ has an upper bound of the same order of magnitude assuming the Generalized Riemann Hypothesis. 
\newline

Progress on unconditional improvements to~\eqref{eq:SchiN} has  two main themes. The first aims at improving the asymptotic size of the constant $c$, thus allowing a $o(1)$ term. The sharpest results in this direction are due to Granville and Soundararajan~\cite{GS}, see also~\cite{Kerr} which deals with the case of arbitrary intervals. The second aims to  determine exactly the $o(1)$ term for a given constant $c$ and we refer the reader to~\cite{Fro,FS,Pomerance} for a series of bounds in this direction. Both problems are known to be closely related to estimating short character sums. A simple way to see this is via Fourier expansion into Gauss sums
$$\chi(n)=\frac{1}{\tau(\chi)}\sum_{\lambda}\chi(\lambda)e_q(\lambda n),$$
which transforms
\begin{align}
\label{eq:basic1}
\sum_{1\le n \le N}\chi(n)\ll \frac{1}{q^{1/2}}\sum_{\lambda}\chi(\lambda)\left(\sum_{1\le n \le N}e_q(\lambda n) \right).
\end{align}
Using the heuristic 
\begin{align*}
\sum_{1\le n \le N}e_q(\lambda n)\sim \begin{cases} N \quad \text{if} \quad \lambda \ll q/N, \\ 0 \quad \text{otherwise}, \end{cases}
\end{align*}
gives 
\begin{align*}
\sum_{1\le n \le N}\chi(n)\ll \frac{N}{q^{1/2}}\sum_{\lambda \ll q/N}\chi(\lambda),
\end{align*}
and hence transforms estimating sums of length $N$ to sums of length $q/N$, an observation which first appears to be due to A. I. Vinogradov~\cite{Vin}. Making the above heuristics rigorous one obtains sums twisted by additive characters 
\begin{align}
\label{eq:partialgaussian}
\sum_{1\le \lambda \le q/N}\chi(\lambda)e_q(a \lambda),
\end{align}
and the constant $c$ in~\eqref{eq:SchiN} which may be obtained by this method depends on how short sums of the form~\eqref{eq:partialgaussian} may be estimated. For example, if for any integer $a$ we have 
\begin{align}
\label{eq:hhhil}
\sum_{1\le n \le N}\chi(n)e_q(a n)=o(N), \quad \text{provided} \quad N\ge q^{\delta},
\end{align}
then the constant $c$ in~\eqref{eq:SchiN} may be taken 
$$c=\delta(1+o(1))/\pi.$$ The details of this argument were first worked out by Hildebrand~\cite{Hild1} and based on ideas of Montgomery and Vaughan~\cite{MV}. One of the key ingredients in is Hildebrand's argument is the Burgess bound~\cite{Bur2,Bur6}, which states that for any primitive $\chi$ mod $q$
\begin{align}
\label{eq:bur}
\sum_{M<n\le M+N}\chi(n) \ll N^{1-1/r}q^{(r+1)/4r^2+o(1)},
\end{align}
provided $r\le 3$ or any $r\ge 2$ if $q$ is cubefree. 
\newline

 In this paper we revisit Hildebrand's argument and obtain the first fully explicit P\'olya-Vinogradov inequality with a constant $c$ below the barrier 
\begin{align*}
c< \begin{cases} \frac{1}{\pi^2}& \mbox{if } \ \ \chi(-1)=1,\\ \frac{1}{2\pi} & \mbox{if } \ \ \chi(-1)=-1, \end{cases}
\end{align*}
which is the limit of previous approaches to an explicit P\'olya-Vinogradov inequality.  The argument of Hildebrand applies a discrete circle method to estimate the sums~\eqref{eq:partialgaussian} and uses the Burgess bound~\eqref{eq:bur} on the major arcs and an estimate of Montgomery and Vaughan~\cite{MV} on the minor arcs. The estimate of Montgomery and Vaughan states that provided the real number $\alpha$ has suitable rational approximation, then we have 
\begin{align}
\label{eq:mv}
\sum_{1\le n \le N}\chi(n)e(\alpha n)\ll \frac{N}{\log{N}},
\end{align}
which allows estimation of very short ranges of the parameter $N$. The bottleneck in the argument is the minor arcs which use the Burgess bound and provide a nontrivial estimate for at best $N\ge q^{1/4+o(1)}$. For an explicit variant of Hildebrand's result, the bottleneck switches from the Burgess bound to Montgomery and Vaughan's estimate. This can be seen by comparing the power saving in the Burgess bound with only a logarithmic factor in~\eqref{eq:mv}. In order to avoid this difficulty we consider an approach which appeals directly to estimates for partial Gaussian sums, which are defined as sums of the form
\begin{align}
\label{eq:hhhil}
\sum_{1\le n \le N}\chi(n)e_q(a n).
\end{align}
These sums were first considered by Burgess for prime modulus~\cite{Bur3} and extended to composite and prime power modulus in~\cite{Bur4,Bur5} for some restricted ranges of parameters. In this paper we extend and make explicit the results of Burgess for the case of squarefree modulus. This requires obtaining uniform estimates for the mean values 
\begin{align*}
\sum_{\mu=1}^{q}\sum_{\lambda=1}^{q}\left| \sum_{1\le v \le V}\chi(\lambda+v)e_q(\mu v)\right|^{2r}.
\end{align*}
Following the approach of Burgess, we reduce to counting lattice points in certain convex bodies averaged over a family of lattices and our main novelty is to appeal to transference principles from the geometry of numbers.
\newline

Fully explicit P\'{o}lya-Vinogradov inequalities have previously been considered by Frolenkov~\cite{Fro}, Frolenkov and Soundararajan ~\cite{FS} and Pomerance~\cite{Pomerance}. The current sharpest result is Frolenkov and Soundararajan~\cite{FS} which states that for all primitive characters $\chi$ we have
\begin{align}
\label{FS}
S(\chi) \le \begin{cases} \frac{1}{\pi^2}\sqrt{q} \log q + \frac{1}{2}\sqrt{q}& \mbox{if } \ \ \chi(-1)=1, ~ q\ge 1200,\\ \frac{1}{2\pi} \sqrt{q} \log q + \sqrt{q} & \mbox{if } \ \ \chi(-1)=-1,~ q \ge 40.  \end{cases}.
\end{align}
 Our main result  is an improvement on~\eqref{FS} for large $q$. A simplified statement of our Theorem~\ref{thm:main} where some accuracy is lost in the secondary terms and range of parameters  is the following.
\begin{cor}
\label{cor:main1}
Let $\ell\ge 2$ be an integer. Suppose $q$ is squarefree and satisfies
\begin{align*}
\log q \ge e^{1088\ell^2}.
\end{align*}
Then for any primitive character $\chi$ mod $q$ we have 
  \begin{align*}
  S(\chi) \le
  \begin{cases}
     \frac{2}{\pi^2}(\frac{1}{4}+\frac{1}{4\ell}) \sqrt{q}\log q +\left(6.5+\frac{1}{1088 \ell}\right)\sqrt{q} \ \  ~~\text{if}~~\ \chi(-1)=1,\\
    ~\\
   \frac{1}{\pi} (\frac{1}{4}+\frac{1}{4\ell}) \sqrt{q} \log q+ \left(6.5+\frac{1}{1088 \ell}\right)\sqrt{q} \ \ ~~\text{if}~~ \ \chi(-1)=-1.
  \end{cases}
\end{align*}
\end{cor} 
A key tool to obtain the above result is an explicit bound for partial Gaussian sums.
\begin{theorem}
\label{thm:main1}
Let $q$ be squarefree and $\chi$ a primitive character mod $q$. For any $M,N$ and $q$ satisfying
\begin{align}
\label{eq:NMaincond}
N\le q^{1/2+1/4(r-1)},
\end{align}
\begin{align}
\label{eq:qMaincond}
q\ge \left(\frac{q}{\phi(q)}\right)^42^{4\omega(q)-4},
\end{align}
and any $0\le a \le q-1$ we have 
\begin{align*}
\left|\sum_{M<n \le M+N}\chi(n)e_q(an)\right|\le 2\Delta_rq^{1/4(r-1)}N^{1-1/r},
\end{align*}
where $\Delta_r$ is given by
$$\Delta_r=2^{6(1+1/r)}(2r)^{\omega(q)/2r}\tau(q)\left(\frac{q}{\phi(q)}\right)^{1/r}(\log{q})^{1/2r}.$$
\end{theorem}
We first note combining Theorem~\ref{thm:main1} with explicit estimates for arithmetic functions gives the following.
\begin{cor}
\label{cor:main2}
Let $q$ be an integer, $\chi$ a primitive character mod $q$ and $a$ any integer. For any integer $\ell \ge 1$ if 
$$N\ge q^{1/4+1/4\ell}$$
and 
$$\log q \ge e^{16\ell^2},$$
then we have 
\begin{align*}
\left|\sum_{n\le N}\chi(n)e_q(an) \right|\le \frac{2^{7}(\log{q})^{1/4\ell}(\log\log{q})^{1/4\ell}}{q^{1/(16\ell^2+8\ell)-1.4/\log\log{q}}}N.
\end{align*}
\end{cor}
For sufficiently large $q$, Corollary~\ref{cor:main2} gives a power saving.
\begin{cor}
\label{cor:main2}
Let $q$ be an integer, $\chi$ a primitive character mod $q$ and $a$ any integer. Let $\alpha > 24$ be a real number and $\ell \ge 1$ be an integer. If $q$ and $N$ satisfy
$$N\ge q^{1/4+1/4\ell}, \quad \log{q}\ge e^{\frac{(\ell^2+\ell/2) \alpha}{(\alpha -16)\ell^2-8\ell} 22.4 \ell^2},$$
then we have 
\begin{align*}
\left|\sum_{n\le N}\chi(n)e_q(an) \right|\le \frac{2^{7}(\log{q})^{1/4\ell}(\log\log{q})^{1/4\ell}}{q^{1/\alpha \ell^2}}N.
\end{align*}
\end{cor}
We will use Corollary~\ref{cor:main2} to show. 
\begin{theorem}
\label{thm:main}
Let $\alpha >24$ be a real number and $\ell \ge 2$ an integer. Suppose is $q$ is squarefree and satisfies 
$$ \log q \ge e^{\frac{(\ell^2+\ell/2) \alpha}{(\alpha -16)\ell^2-8\ell} 22.4 \ell^2}.$$
For any primitive character $\chi$ mod $q$ and integer $N$ we have 
\begin{align*}
  \left|\sum_{1\le n \le N}\chi(n) \right| \le   
  \end{align*}
  \begin{align*}
  \begin{cases}
     \frac{2}{\pi^2}(\frac{1}{4}+\frac{1}{4\ell}) \sqrt{q}\log q +\left(6.5+\frac{2^{9}(\log{q})^{1/4\ell}(\log\log{q})^{1/4\ell}\log q}{\pi q^{ 1/\alpha \ell^2}}\right)\sqrt{q} \ \ ~~\text{if}~~ \ \ \chi(-1)=1,\\
    ~\\
   \frac{1}{\pi} (\frac{1}{4}+\frac{1}{4\ell}) \sqrt{q} \log q+ \left(6.5+\frac{2^{9}(\log{q})^{1/4\ell}(\log\log{q})^{1/4\ell}\log q}{\pi q^{1/\alpha \ell^2}}\right)\sqrt{q} \ \ ~~\text{if}~~ \ \ \chi(-1)=-1.
  \end{cases}
\end{align*}
\end{theorem}
Theorem~\ref{thm:main} has applications to the theory of Dirichlet $L$-functions. In Section~\ref{sec:siegel} we give a new explicit estimate for exceptional zeros.
\newline

{\bf Acknowledgement:} The authors would like to thank Tim Trudgian for a number of useful discussions. 
\section{Preliminary estimates for arithmetic function}
\label{sec:preliminaryexplicit}
In this section we collect some well known estimates for arithmetic functions. For a proof of the following, see~\cite[Theorem~12]{Robin1}.
\begin{lemma}
\label{lem:omegaub}
For any integer $n\ge 3$ we have 
\begin{align*}
\omega(n)\leq \frac{\log n}{\log \log n}+1.45743\frac{\log n}{(\log \log n)^2}.
\end{align*}
\end{lemma}
For a proof of the following, see~\cite[Theorem~1]{Robin}.
\begin{lemma}
\label{lem:tauub}
For any integer $n\ge 3$ we have 
\begin{align*}
\tau (n) \leq e^{\frac{1.5379 \log 2 \log n}{\log \log n}}.
\end{align*}
\end{lemma}
For a proof of the following see~\cite[Theorem~15]{Rosser}.
\begin{lemma}
\label{lem:phiUN}
For any integer $n\ge 3$ we have 
\begin{align*}
\phi (n) > \frac{n e^{-\gamma}}{\log \log n + \frac{2.50637}{\log \log n}}.
\end{align*}
\end{lemma}
For a proof of the following, see~\cite[Lemma~2]{Pomerance}. 
\begin{lemma}
\label{LI2}
Uniformly for $x \geqslant 1$ and  $\alpha \in \R$ we have
\begin{equation}
\sum_{n \leq x}\frac{1-\cos (\alpha n)}{n} \leq \log x +\gamma +\log 2 +\frac{3}{x},
\end{equation}
and
\begin{equation}
\sum_{n \leq x}\frac{|\sin (\alpha n)|}{n} \leq \frac{2}{\pi} \log x +\frac{2}{\pi}\left( \gamma +\log 2 +\frac{3}{x}\right).
\end{equation}
\end{lemma}
We use notation $O^*$ in a similar way to $O$ notation with implied constant $1$. For example 
$$f=O^{*}(g) \quad \text{if and only if} \quad |f|\le g.$$
 The following is a well known consequence of the sieve of Eratosthenes.
\begin{lemma}
\label{lem:SE}
For any integers $q$ and $U$ we have 
\begin{align*}
\sum_{\substack{1\le u \le U \\ (u,q)=1}}1=\frac{\phi(q)}{q}U+O^{*}(2^{\omega(q)}).
\end{align*}
\end{lemma}
The proof of the following is the same as~\cite[Lemma~1]{Trev} which deals with the case $q=p$ prime.
\begin{lemma}
\label{lem:multcong}
For integers $q,M,N,U$ satisfying
\begin{align*}
24 \le U\le \frac{N}{12},
\end{align*}
let $I_q(N,U)$ count the number of solutions to the congruence
$$n_1u_1\equiv n_2u_2 \mod{q},$$
with variables satisfying
$$M\le n_1,n_2\le M+N, \quad 1\le u_1,u_2\le U, \quad (u_1u_2,q)=1.$$
We have 
\begin{align*}
I_q(N,U)\le  2UN\left(\frac{NU}{q}+\log(1.85 U) \right).
\end{align*}
\end{lemma}

\section{Background from the geometry of numbers}
 The following is Minkowski's second theorem, for a proof see~\cite[Theorem~3.30]{TV}.
\begin{lemma}
\label{lem:mst}
Suppose $\Gamma \subseteq \R^{d}$ is a lattice, $D\subseteq \R^{d}$ a convex body and let $\lambda_1,\dots,\lambda_d$ denote the successive minima of $\Gamma$ with respect to $D$. Then we have
$$\frac{1}{\lambda_1\dots\lambda_d}\le \frac{d!}{2^d}\frac{\text{Vol}(D)}{\text{Vol}(\R^d/\Gamma)}.$$
\end{lemma}
 For a proof of the following,  see~\cite[Proposition~2.1]{BHM}.
\begin{lemma}
\label{lem:latticesm}
Suppose $\Gamma \subseteq \R^{d}$ is a lattice, $D\subseteq \R^{d}$ a convex body and let $\lambda_1,\dots,\lambda_d$ denote the successive minima of $\Gamma$ with respect to $D$. Then we have
$$|\Gamma \cap D|\le \prod_{j=1}^{d}\left(\frac{2j}{\lambda_j}+1\right).$$
\end{lemma}
For a lattice $\Gamma$ and  a convex body $D$  we define the dual  lattice $\Gamma^*$ and dual body $D^*$ by
\begin{align}
\label{eq:Gammstar}
\Gamma^*=\{ x\in \R^{d} : \langle x,y \rangle \in \Z \quad \text{for all} \quad y\in \Gamma\},
\end{align}
\begin{align}
\label{eq:Dstar}
D^{*}=\{ x\in \R^{d} : \langle x,y \rangle \le 1 \quad \text{for all} \quad y\in D\}.
\end{align}

The following is known as a transference theorem and is due to Mahler~\cite{Mah}.
\begin{lemma}
\label{lem:transfer}
Let $\Gamma\subset \R^{d}$ be a lattice, $D\subseteq \R^{d}$ a symmetric convex body and let $\Gamma^{*}$ and $D^{*}$ denote the dual lattice and dual body. Let $\lambda_1,\dots,\lambda_d$ denote the  successive minima of $\Gamma$ with respect to $D$ and $\lambda_1^{*},\dots,\lambda_d^{*}$ the successive minima of $\Gamma^{*}$ with respect to $D^{*}$. For each $1\le j \le d$ we have
$$\lambda_j \lambda^*_{d-j+1}\le (n!)^2.$$
\end{lemma} 
\section{Mean value estimates}
Our next result follows from the argument of~\cite[Lemma~7]{Bur1}.
\begin{lemma}
\label{lem:squarefree}
Let $q$ be squarefree and $\chi$ a primitive character mod $q$. Suppose the tuple of integers $v=(v_1,\dots,v_{2r})$ satisfies
$$|\{v_1,\dots,v_{2r}\}|\ge r+1,$$
and for each $1\le j\le 2r$ define
$$A_j(v)=\prod_{i\neq j}(v_i-v_j).$$
There exist some $1\le j \le 2r$ such that 
$$A_j(v)\neq 0,$$ 
and 
\begin{align*}
\left|\sum_{\lambda=1}^{q}\chi\left(F_v(\lambda)\right)\right|\le (2r)^{\omega(q)}(q,A_j(v))^{1/2}q^{1/2},
\end{align*}
where 
$$F_v(\lambda)=\frac{(\lambda+v_1)\dots (\lambda+v_{r})}{(\lambda+v_{r+1})\dots (\lambda+v_{2r})}.$$
\end{lemma}
\begin{lemma}
\label{lem:mvsquarefree}
Let $q$ be squarefree and $\chi$ a primitive character mod $q$. For any integer $r\ge 2$, real number $1\le V< q$ and sequence of complex numbers $\beta_v$ satisfying 
$$|\beta_v|\le 1,$$
we have 
\begin{align*}
&\frac{1}{q}\sum_{\lambda=1}^{q}\sum_{\mu=1}^{q}\left|\sum_{1\le v \le V}\beta_v \chi(\lambda+v)e_q(\mu v) \right|^{2r}\le r!qV^{r} 
+2^{2r}r(2r)^{\omega(q)}\tau(q)^{2r}q^{1/2}V^{2r-1} \\ & \quad \quad \quad \quad+4^{2r+1}r(2r)^{\omega(q)}(2r-1)!^2\tau(q)^{2r}q^{1/2}V^{2r-3/2}.
\end{align*}
\end{lemma}
\begin{proof}
Let 
$$S=\sum_{\lambda=1}^{q}\sum_{\mu=1}^{q}\left|\sum_{1\le v \le V}\beta_v \chi(\lambda+v)e_q(\mu v) \right|^{2r}.$$
Expanding the $2r$-th power, interchanging summation gives  and using the assumption $|\beta_v|\le 1$ we get
\begin{align*}
S&\le \sum_{1\le v_1,\dots,v_{2r}\le V}\left|\sum_{\lambda=1}^{q}\chi(F_v(\lambda)) \right|\left|\sum_{\mu=1}^{q}e_q(\mu(v_1+\dots-v_{2r})) \right| \\
&=q\sum_{\substack{1\le v_1,\dots,v_{2r}\le V \\ v_1+\dots+v_r=v_{r+1}+\dots+v_{2r}}}\left|\sum_{\lambda=1}^{q}\chi(F_v(\lambda)) \right|,
\end{align*}
where $F_v$ is defined as in Lemma~\ref{lem:squarefree}. We partition summation over $v_1,\dots,v_{2r}$ into sets 
\begin{align*}
\cV_1&=\{ (v_1,\dots,v_{2r})\in [1,V]^{2r} \ : \ |\{v_1,\dots, v_{2r}\}|\le r \}, \\
\cV_2&=\{ (v_1,\dots,v_{2r})\in [1,V]^{2r} \ : \ v_1+\dots-v_{2r}=0, \ (v_1,\dots,v_{2r})\not \in \cV_1 \},
\end{align*}
and note that 
$$\{ (v_1,\dots,v_{2r})\in [1,V]^{2r} \ : \ v_1+\dots-v_{2r}=0 \}\subseteq \cV_1\cup \cV_2.$$
For tuples in $\cV_1$ we use the trivial bound 
$$\left|\sum_{\lambda=1}^{q}\chi(F_v(\lambda)) \right|\le q,$$
to get 
\begin{align*}
S\le q^2|\cV_1|+q\sum_{(v_1,\dots,v_{2r})\in \cV_2}\left|\sum_{\lambda=1}^{q}\chi(F_v(\lambda)) \right|.
\end{align*}
If $(v_1,\dots,v_{2r})\in \cV_1$ then fixing values $v_1,\dots,v_r$ with $V^r$ choices gives $r!$ possible values for remaining $v_{r+1},\dots,v_{2r}$ and hence
\begin{align*}
S\le r!q^2V^{r}+q\sum_{(v_1,\dots,v_{2r})\in \cV_2}\left|\sum_{\lambda=1}^{q}\chi(F_v(\lambda)) \right|.
\end{align*}
We partition 
$$\cV_2\subseteq \bigcup_{j=1}^{2r}\cV_{2,j},$$
where 
\begin{align*}
\cV_{2,j}=\{ v=(v_1,\dots,v_{2r})\in \cV_2 \ : \ A_j(v)\neq 0\},
\end{align*}
and $A_j(v)$ is defined as in Lemma~\ref{lem:squarefree}. This implies that 
\begin{align}
\label{eq:SSj}
S\le r!q^2V^{r}+q\sum_{j=1}^{2r}S_j\le r!q^2V^2+2rqS_1,
\end{align}
where we define
\begin{align*}
S_j=\sum_{(v_1,\dots,v_{2r})\in \cV_{2,j}}\left|\sum_{\lambda=1}^{q}\chi(F_v(\lambda)) \right|,
\end{align*}
and use symmetry to estimate 
$$\sum_{j=1}^{2r}S_j\le 2r \max_{j}S_j \le 2rS_1.$$
Considering $S_1$, by Lemma~\ref{lem:squarefree} we have 
\begin{align*}
S_1&\le (2r)^{\omega(q)}q^{1/2}\sum_{(v_1,\dots,v_{2r})\in \cV_{2,1}}(q,A_j(v))^{1/2} \\
&\le (2r)^{\omega(q)}q^{1/2}\sum_{d|q}d^{1/2}|\cV_{2,1}(d)|,
\end{align*}
with $\cV_{2,1}(d)$ defined by 
\begin{align*}
\cV_{2,1}(d)=\{ (v_1,\dots,v_{2r})\in \cV_{2,1} \ : \ (q,A_1(v))=d \}.
\end{align*}
Isolating the contribution from $d=1$ and using that  
$$|\cV_{2,1}(1)|\le |\cV_{2,1}|\le V^{2r-1},$$ gives 
\begin{align}
\label{eq:S1B}
S_1\le (2r)^{\omega(q)}q^{1/2}V^{2r-1}+(2r)^{\omega(q)}q^{1/2}\sum_{\substack{d|q \\ d>1}}d^{1/2}|\cV_{2,1}(d)|.
\end{align}
Fix some $d|q$ and consider $|\cV_{2,1}(d)|$. 
If $(v_1,\dots,v_{2r})\in \cV_{2,1}(d),$ then 
$$d=(q,A_j(v))\le \prod_{j=2}^{2r-1}\left(v_{1}-v_j,q\right),$$
and hence there exists $d_2,\dots,d_{2r}$ such that 
$$d_2\dots d_{2r}=d,$$
and 
\begin{align}
\label{eq:vvv1}
v_i\equiv v_1 \mod{d_i}, \quad 2\le i \le 2r,
\end{align}
and we may have $d_i=1$ for some values of $i$. Let  $\cV_2(d_2,\dots,d_{2r})$ count the number of $(v_1,\dots,v_{2r})$ satisfying~\eqref{eq:vvv1} and 
\begin{align}
\label{eq:vvv2}
v_1+\dots-v_{2r}=0, \quad 1\le v_1,\dots,v_{2r}\le V.
\end{align}
If $v=(v_1,\dots,v_{2r})\in \cV_2(d_2,\dots,d_{2r})$ then $A_1(v)\neq 0$ and hence if $\cV_2(d_2,\dots,d_{2r})\neq 0$ then by~\eqref{eq:vvv1} we must have  each $d_2,\dots,d_{2r}\le V$. This implies that 
\begin{align}
\label{eq:VDD}
|\cV_{2,1}(d)|\le \sum_{\substack{d_2,\dots,d_{2r}\le V \\ d_2\dots d_{2r}=d}}|\cV_2(d_2,\dots,d_{2r})|.
\end{align}
Fix some  $d_2,\dots,d_{2r}$  satisfying 
$$d_2\dots d_{2r}=d, \quad d_2,\dots,d_{2r}\le V,$$
and consider $\cV_2(d_2,\dots,d_{2r})$. Each $v_i$ may be written in the form
$$v_i=v_1+d_ih_i, \quad |h_i|\le V/d_i.$$
If $(v_1,\dots,v_{2r})\in \cV_{2}(d_2,\dots,d_{2r})$ then
$$v_1+\dots-v_{2r}=0,$$
which implies that 
\begin{align}
\label{eq:1eqnV2}
d_2h_2+\dots-d_{2r}h_{2r}=0.
\end{align}
Let $T_3(d_2,\dots,d_{2r})$ count the number of solutions to the equation~\eqref{eq:1eqnV2} with variables satisfying
\begin{align*}
|h_i|\le V/d_i, \quad 2\le i \le 2r.
\end{align*}
Fixing $v_1$ with at most $V$ choices in $\cV_2(d_2,\dots,d_{2r})$, we see that 
\begin{align}
\label{eq:Vddd}
|\cV_2(d_2,\dots,d_{2r})|\le VT_3(d_2,\dots,d_{2r}),
\end{align}
and hence it remains to estimate $T_3(d_2,\dots,d_{2r})$. We consider two cases. First suppose that there exists some $2\le i \le 2r$ such that $d_i=1$. Fixing variables $v_2,\dots,v_{i-1},v_{i+1},\dots,v_{2r}$ with at most 
$$2^{2r-2}\prod_{j\neq i}\frac{V}{d_j},$$
values gives at most $1$ solution to~\eqref{eq:1eqnV2} in remaining variable $v_i$. This implies that 
\begin{align*}
T_3(d_2,\dots,d_{2r})\le 2^{2r-2}V^{2r-2}\prod_{j\neq i}\frac{1}{d_j}=\frac{2^{2r-2}V^{2r-2}}{d},
\end{align*}
where we have used $d_i=1$ to get equality $d=d_2\dots d_{i-1}d_{i+1}\dots d_{2r}$. By~\eqref{eq:Vddd} this implies that 
\begin{align*}
|\cV_2(d_2,\dots,d_{2r})|\le \frac{2^{2r-2}V^{2r-1}}{d}.
\end{align*}
Consider next when $d_i\neq 1$ for all $2\le i \le 2r$. By permuting terms we may suppose that that $d_{2}$ is minimum and  $d_{2r-1}$ is maximum. This permutation will not affect the constants in our final bound since we will give an estimate independent of $d_2$ and $d_{2r}$. By the pigeonhole principle  
\begin{align}
\label{eq:d2UB}
d_{2}\le d^{1/(2r-1)}.
\end{align}
Considering the equation~\eqref{eq:1eqnV2}, each value of $h_2,\dots,h_{2r-1}$ gives at most one solution in variable $h_{2r}$. Since $h_2,\dots,h_{2r-1}$ must satisfy
\begin{align*}
d_2h_2+\dots-d_{2r-1}h_{2r-1}\equiv 0 \mod{d_{2r}},
\end{align*}
 defining the lattice 
\begin{align*}
\cL=\{(h_2,\dots,h_{2r-1})\in \Z^{2r-2} \ : d_2h_2+\dots-d_{2r-1}h_{2r-1}\equiv 0 \mod{d_{2r}}\},
\end{align*}
and the convex body 
\begin{align*}
D=\left\{(h_2,\dots,h_{2r-1})\in \R^{2r-2} \ : \ |h_i|\le \frac{2V}{d_i}\right\},
\end{align*}
we have 
\begin{align}
\label{eq:V2b1123}
T_3(d_2,\dots,d_{2r})\le |\cL\cap D|.
\end{align}
Considering $\cL\cap D$, if 
\begin{align}
\label{eq:Vd2r}
\frac{V}{d_2}\ge d_{2r},
\end{align}
then choosing  $h_3,\dots,h_{2r-1}$ with at most 
\begin{align}
\label{eq:prelimstep1}
\prod_{j=3}^{2r-1}\frac{2V}{d_j},
\end{align}
values gives a congruence of the form
$$d_2h_2\equiv c \mod{d_{2r}},$$
in remaining variable $h_2$. Since $d_2|q,$ $d_{2r-1}|q$ and $q$ is squarefree, we must have $(d_2,d_{2r})=1$ and hence $h_2$ is uniquely determined mod $d_{2r}$. The condition~\eqref{eq:Vd2r} implies at most 
\begin{align*}
\frac{2V}{d_2d_{2r}},
\end{align*}
values for $h_2$ and hence by~\eqref{eq:prelimstep1} and~\eqref{eq:V2b1123}
\begin{align}
\label{eq:smalld1d2}
T_3(d_2,\dots,d_{2r})\le 2^{2r-2}\frac{V^{2r-2}}{d},
\end{align}
which by~\eqref{eq:Vddd} implies 
\begin{align}
\label{eq:smallcase}
|\cV_2(d_2,\dots,d_{2r})|\le 2^{2r-2}\frac{V^{2r-1}}{d} \quad \text{provided} \quad d_2d_{2r}\le V.
\end{align}
It remains to consider when
\begin{align*}
V\le d_2d_{2r}.
\end{align*}
Let $\lambda_1,\dots,\lambda_{2r-2}$ denote the successive minima of $\cL$ with respect to $D$. We consider two cases depending on the value of $\lambda_{2r-2}.$ Suppose first that 
\begin{align}
\label{eq:minimacase1}
\lambda_{2r-2}\le 1.
\end{align}
By Lemma~\ref{lem:mst} and Lemma~\ref{lem:latticesm} we have 
\begin{align*}
|\cL\cap D|\le 4^{2r-2}(2r-2)!\frac{1}{\lambda_1\dots \lambda_{2r-2}}\le 2^{2r-2}(2r-2)!^2\frac{\text{Vol}(D)}{\text{Vol}(\R^{2r-2}/\cL)}.
\end{align*}
Since 
\begin{align*}
\text{Vol}(D)=\frac{2^{2r-2}V^{2r-2}}{d_2\dots d_{2r-1}}, \quad \text{Vol}(\R^{2r-2}/\cL)=d_{2r},
\end{align*}
and using that $d_2\dots d_{2r}=d$, we get 
\begin{align*}
|\cL\cap D|\le 4^{2r-2}(2r-2)!^2\frac{V^{2r-2}}{d},
\end{align*}
which combined with~\eqref{eq:Vddd} and~\eqref{eq:V2b1123}  gives 
\begin{align}
\label{eq:mcase1}
|\cV_2(d_2,\dots,d_{2r})|\le 4^{2r-2}(2r-2)!^2\frac{V^{2r-1}}{d} \quad \text{provided} \quad \lambda_{2r-2}\le 1.
\end{align}
Suppose next that
\begin{align}
\label{eq:minimacase1}
\lambda_{2r-2}> 1.
\end{align}
By Lemma~\ref{lem:transfer}
\begin{align*}
\lambda_{1}^{*}\le (2r-2)!^2,
\end{align*}
and hence 
\begin{align}
\label{eq:dualLD}
|\cL^{*}\cap (2r-2)!^2\cD^{*}|\neq \{0\}.
\end{align}
 Recalling the definitions~\eqref{eq:Gammstar} and~\eqref{eq:Dstar} we calculate 
\begin{align*}
\cL^{*}=\left\{ \left(\frac{y_2}{d_{2r}},\dots,\frac{y_{2r-1}}{d_{2r}}\right)\in \Z^{2r-2}/d_{2r} \ : \ \exists \ \lambda\in \Z, \ \  \   d_j\lambda \equiv y_j \mod{d_{2r}} \right\},
\end{align*}
and 
\begin{align*}
D^{*}=\left\{ (z_2,\dots,z_{2r-1})\in \R^{2r-2} \ : \frac{2V}{d_2}|z_2|+\dots+\frac{2V}{d_{2r-1}}|z_{2r-1}|\le 1  \right\}.
\end{align*}
The above combined with~\eqref{eq:dualLD} implies there exists some $1\le \lambda \le d_{2r}-1$ satisfying 
\begin{align}
\label{eq:LLLLL}
\lambda d_j \equiv y_j \mod{d_{2r}}, \quad |y_j|\le (2r-2)!^2\frac{d_{2r}d_j}{2V}, \quad 2\le j \le 2r-1.
\end{align}
Since $q$ is squarefree, the $d_i$'s are pairwise coprime and hence
$$\lambda d_j \not \equiv 0 \mod{d_{2r}}.$$
Returning to the intersection $\cL\cap D$, if $h_2,\dots,h_{2r-1}$ satisfy 
\begin{align}
\label{eq:d22222}
d_2h_2+\dots-d_{2r-1}h_{2r-1}\equiv 0 \mod{d_{2r}}, \quad |h_j|\le \frac{V}{d_j},
\end{align}
then by~\eqref{eq:LLLLL} we must have 
\begin{align*}
y_2h_2+\dots+y_{2r-1}h_{2r-1}=kd_{2r},
\end{align*}
for some $k\in \Z$. From~\eqref{eq:LLLLL} and~\eqref{eq:d22222}
$$|y_2h_2+\dots-y_{2r-1}h_{2r-1}|\le (2r-1)!^2d_{2r},$$
and hence there are at most $(2r-1)!^2$ possible values of $k$. For each such value of $k$ we choose variables $h_3,\dots,h_{2r-1}$ with at most 
$$2^{2r-1}\prod_{j=3}^{2r-1}\frac{V}{d_{j}},$$
values to get at most one remaining value of $h_2$. This implies that 
\begin{align*}
|\cL\cap D|\le 2^{2r-1}(2r-1)!^2\prod_{j=3}^{2r-1}\frac{V}{d_{j}},
\end{align*}
and hence by~\eqref{eq:d2UB}
\begin{align*}
|\cV_2(d_2,\dots,d_{2r})|\le 2^{2r-1}(2r-1)!^2\frac{V^{2r-2}}{d^{1-1/(2r-1)}}.
\end{align*}
Combining with~\eqref{eq:mcase1}, we get 
\begin{align*}
|\cV_2(d_2,\dots,d_{2r})|\le 4^{2r-2}(2r-1)!^2\left(\frac{V^{2r-1}}{d}+\frac{V^{2r-2}}{d^{1-1/(2r-1)}}\right),
\end{align*}
which gives our final estimate for $\cV_2(d_2,\dots,d_{2r})$ in the case that $d_j\neq 1$ for each $j$. By~\eqref{eq:VDD} 
\begin{align}
\label{eq:caselarge123}
\nonumber |\cV_{2,1}(d)|&\le \tau(d)^{2r-1}\max_{d_2\dots d_{2r}=d}|\cV_2(d_2,\dots,d_{2r})| \\ &\le 4^{2r-1}(2r-1)!^2\tau(d)^{2r-1}\left(\frac{V^{2r-1}}{d}+\frac{V^{2r-2}}{d^{1-1/(2r-1)}}\right).
\end{align}
We also note that~\eqref{eq:smallcase} implies the estimate 
\begin{align}
\label{eq:casesmall123}
 |\cV_{2,1}(d)|&\le \tau(d)^{2r-1}2^{2r-2}\frac{V^{2r-1}}{d}, \quad d\le V.
\end{align}
Using the above in~\eqref{eq:S1B}, we see that 
\begin{align*}
&S_1\le (2r)^{\omega(q)}q^{1/2}V^{2r-1}+(2r)^{\omega(q)}q^{1/2}\tau(q)^{2r-1}2^{2r-2}V^{2r-1}\sum_{\substack{d|q \\ 1<d\le V}}\frac{1}{d^{1/2}} \\
&+4^{2r-1}(2r)^{\omega(q)}(2r-1)!^2\tau(q)^{2r-1}q^{1/2}\sum_{\substack{d|q \\ d\ge V}}\left(\frac{V^{2r-1}}{d^{1/2}}+\frac{V^{2r-2}}{d^{1/2-1/(2r-1)}}\right),
\end{align*}
which simplifies to
\begin{align*}
S_1\le 2^{2r-1}(2r)^{\omega(q)}\tau(q)^{2r}q^{1/2}V^{2r-1}+4^{2r}(2r)^{\omega(q)}(2r-1)!^2\tau(q)^{2r}q^{1/2}V^{2r-3/2},
\end{align*}
and the result follows combining with~\eqref{eq:SSj}.
\end{proof}
Using the estimates from Section~\ref{sec:preliminaryexplicit} we may put the bound of Lemma~\ref{lem:mvsquarefree} in the following simpler form.
\begin{cor}
\label{cor:mvsquarefree}
Let $q$ be squarefree and $\chi$ a primitive character mod $q$. For any integer $r$ and real number $V$ satisfying
$$r\ge 2, \quad 1\le V< q, \quad V\ge (2r-1)!^2,$$ and sequence of complex numbers $\beta_v$ satisfying 
$$|\beta_v|\le 1,$$
we have 
\begin{align*}
&\frac{1}{q}\sum_{\lambda=1}^{q}\sum_{\mu=1}^{q}\left|\sum_{1\le v \le V}\beta_v \chi(\lambda+v)e_q(\mu v) \right|^{2r}\le r!qV^{r} 
+4^{4r}(2r)^{\omega(q)}\tau(q)^{2r}q^{1/2}V^{2r-1}.
\end{align*}
\end{cor}
\begin{proof}
Let 
$$\Sigma_2=2^{2r}r(2r)^{\omega(q)}\tau(q)^{2r}q^{1/2}V^{2r-1},$$
and 
$$\Sigma_3=4^{2r+1}r(2r)^{\omega(q)}(2r-1)!^2\tau(q)^{2r}q^{1/2}V^{2r-3/2}.$$
By Lemma~\ref{lem:mvsquarefree} it is sufficient to show that 
\begin{align*}
\Sigma_2+\Sigma_3\le 4^{4r}(2r)^{\omega(q)}\tau(q)^{2r}q^{1/2}V^{2r-1} .
\end{align*}
Using 
$$2^{2r}\le 4^{2r+1}, \quad r\le 4^r,$$
we have 
\begin{align*}
\Sigma_2+\Sigma_3\le 4^{3r+1}(2r)^{\omega(q)}\tau(q)^{2r}\left(1+\frac{(2r-1)!^2}{V^{1/2}}\right)q^{1/2}V^{2r-1}.
\end{align*}
The assumption $V\ge (2r-1)!^2$ implies that 
\begin{align*}
\Sigma_2+\Sigma_3\le 4^{3r+2}(2r)^{\omega(q)}\tau(q)^{2r}q^{1/2}V^{2r-1}\le 4^{4r}(2r)^{\omega(q)}\tau(q)^{2r}q^{1/2}V^{2r-1},
\end{align*}
and completes the proof.
\end{proof}
\section{Proof of Theorem~\ref{thm:main1}}
We first introduce the notation
\begin{align}
\label{eq:thm1deltadef}
\Delta_r =2^{6(1+1/r)}(2r)^{\omega(q)/2r}\tau(q)\left(\frac{q}{\phi(q)}\right)^{1/r}(\log{q})^{1/2r}.
\end{align}

We proceed by induction on $N$ and formulate our induction hypothesis as follows: For any integer $K<N$ and arbitrary $M$ we have 
\begin{align*}
\left|\sum_{M<n \le M+K}\chi(n)e_q(an)\right|\le 2\Delta_rq^{1/4(r-1)}K^{1-1/r}.
\end{align*}
Since the estimate is trivial for $N\le q^{1/4}$ this forms the basis of our induction.  Considering the sum
\begin{align*}
S=\sum_{M<n \le M+N}\chi(n)e_q(an),
\end{align*}
since for any integer $h<N$ the difference 
\begin{align*}
S-\sum_{M<n \le M+N}\chi(n+h)e_q(a(n+h)),
\end{align*}
splits as two sums of length $\le h$, by our induction hypothesis
\begin{align}
\label{eq:indH}
\left|S-\sum_{M<n \le M+N}\chi(n+h)e_q(a(n+h))\right|\le 4\Delta_r q^{1/4(r-1)}h^{1-1/r} .
\end{align}
Define
\begin{align}
\label{eq:UVdef}
V=\left\lfloor rq^{1/2(r-1)} \right\rfloor, \quad U=\left\lfloor \frac{N}{16 r q^{1/2(r-1)}}\right\rfloor,
\end{align}
so that 
\begin{align}
\label{eq:UVub}
UV<\frac{N}{16},
\end{align}
and let $\cU$ denote the set 
\begin{align*}
\cU=\{1\le u \le U \ : \ (u,q)=1\}.
\end{align*}
Averaging~\eqref{eq:indH} over values of the form 
$$h=uv, \quad 1\le v \le V, \quad u\in \cU,$$
gives 
\begin{align}
\label{eq:SSS}
S\le \frac{1}{V|\cU|}W+\frac{\Delta_r}{2}q^{1/4(r-1)}N^{1-1/r},
\end{align}
where 
\begin{align*}
W=\sum_{\substack{M<n\le M+N \\ u\in \cU}}\left|\sum_{1\le v \le V}\chi(n+uv)e_q(a(n+uv))\right|.
\end{align*}
Let $I(\lambda)$ count the number of solutions to the congruence 

\begin{align*}
nu^{-1}\equiv \lambda \mod{q}, \quad u\in \cU, \quad M<n\le M+N.
\end{align*}
Note that 
\begin{align}
\label{eq:Ilambda1}
\sum_{\lambda}I(\lambda)=N|\cU|,
\end{align}
and by Lemma~\ref{lem:multcong}
\begin{align}
\label{eq:Ilambda2}
\sum_{\lambda}I(\lambda)^2\le 2UN\left(\frac{NU}{q}+\log{1.85U}\right).
\end{align}
We have 
\begin{align*}
W&=\sum_{\substack{M<n\le M+N \\ u\in \cU}}\left|\sum_{1\le v \le V}\chi(nu^{-1}+v)e_q(auv)\right| \\
&\le \sum_{\lambda}I(\lambda)\max_{\rho}\left|\sum_{1\le v \le V}\chi(\lambda+v)e_q(\rho v)\right|.
\end{align*}
At this stage we use some ideas of Chamizo~\cite{Cham}. Define
\begin{align}
\label{eq:Ldef}
L=\left\lfloor \frac{q}{4V}-\frac{1}{2} \right\rfloor,
\end{align}
and let 
\begin{align*}
\theta(v)=\frac{\sin \pi v/q}{\sin(\pi(2L+1)v/q)}.
\end{align*}
If $1\le v \le V$ then 
\begin{align}
\label{eq:thetabbound}
|\theta(v)|\le \frac{1}{L+1/2}.
\end{align}
Since 
\begin{align*}
\theta(v)\sum_{|\ell|\le L}e_q(\ell v)=1,
\end{align*}
we have 
\begin{align*}
W&\le \sum_{\lambda}I(\lambda)\max_{\rho}\left|\sum_{1\le v \le V}\theta(v)\sum_{|\ell|\le L}e_q(\ell v)\chi(\lambda+v)e_q(\rho v)\right| \\ 
&\le \sum_{\lambda}\sum_{|\ell|\le L}I(\lambda)\max_{\rho}\left|\sum_{1\le v \le V}\theta(v)\chi(\lambda+v)e_q((\rho+\ell) v)\right|.
\end{align*}
By H\"{o}lder's inequality 
\begin{align}
\label{eq:W2r}
W^{2r}\le (2L)^{2r-1}W_1^{2r-2}W_2W_3,
\end{align}
where 
\begin{align*}
W_1=\sum_{\lambda}I(\lambda),
\end{align*}
\begin{align*}
W_2=\sum_{\lambda}I(\lambda)^2,
\end{align*}
and 
\begin{align*}
W_3=\sum_{\lambda}\sum_{|\ell|\le L}\max_{\rho}\left|\sum_{1\le v \le V}\theta(v)\chi(\lambda+v)e_q((\rho+\ell) v)\right|^{2r}.
\end{align*}
By~\eqref{eq:Ilambda1} and~\eqref{eq:Ilambda2} we have 
\begin{align}
\label{eq:W1W2}
W_1\le N|\cU|, \quad W_2\le 2NU\log(1.85U)\left(1+\frac{NU}{q\log(1.85U)}\right).
\end{align}
Considering $W_3$, extending summation from $|\ell|\le L$ to a complete residue system allows us to remove the maximum over $\rho$ and hence 
\begin{align*}
W_3\le \sum_{\lambda}\sum_{\mu}\left|\sum_{1\le v \le V}\theta(v)\chi(\lambda+v)e_q(\mu v)\right|^{2r}.
\end{align*}
Using~\eqref{eq:Ldef},~\eqref{eq:thetabbound} and Corollary~\ref{cor:mvsquarefree} gives 
\begin{align*}
W_3\le \frac{4^{4r}(2r)^{\omega(q)}}{L^{2r}}\tau(q)^{2r}q\left(r!qV^{r} 
+q^{1/2}V^{2r-1}\right).
\end{align*}
Recalling~\eqref{eq:UVdef}, the above simplifies to
\begin{align}
W_3\le \frac{2^{8r+1}(2r)^{\omega(q)}\tau(q)^{2r}r^{2r}q^{1+(3r-2)/2(r-1)}}{L^{2r}}.
\end{align}
Hence by the above,~\eqref{eq:W2r} and~\eqref{eq:W1W2}
\begin{align}
\label{eq:WUBS}
W^{2r}&\le \frac{2^{10r+2}(2r)^{\omega(q)}\tau(q)^{2r}r^{2r}q^{1+(3r-2)/2(r-1)}(N|\cU|)^{2r-2}\log{q}}{L},  \nonumber
\end{align}
where we have used the following inequalities 
\begin{align*}
1.85 U\le q, \quad |\cU|\le U, \quad NU\le q.
\end{align*}
Recalling~\eqref{eq:Ldef} we have the upper bound 
\begin{align*}
L\ge \frac{q}{8V},
\end{align*}
and hence by the above 
\begin{align*}
W^{2r}\le 2^{10r+5}(2r)^{\omega(q)}\tau(q)^{2r}r^{2r}q^{(3r-2)/2(r-1)}(N|\cU|)^{2r-2}NU V\log{q}.
\end{align*}
This implies
\begin{align*}
\frac{W^{2r}}{|\cU|^{2r}V^{2r}}\le 2^{10r+5}(2r)^{\omega(q)}\tau(q)^{2r}r^{2r}\frac{q^{(3r-2)/2(r-1)}}{V^{2r-1}}\frac{N^{2r-1}U}{|\cU|^2} \log{q}.
\end{align*}
By the condition~\eqref{eq:qMaincond} and Lemma~\ref{lem:SE} we have 
\begin{align*}
|\cU|\ge \frac{\phi(q)}{2q}U,
\end{align*}
which combined with~\eqref{eq:UVdef} gives
\begin{align*}
\frac{N^{2r-1}U}{|\cU|^2}\le 2\left(\frac{q}{\phi(q)}\right)^2\frac{N^{2r-1}}{U}\le 2^6\left(\frac{q}{\phi(q)}\right)^2r N^{2r-2}q^{1/2(r-1)}.
\end{align*}
Using
\begin{align*}
\frac{q^{(3r-2)/2(r-1)}}{V^{2r-1}}\le \frac{2^{2r}}{r^{2r-1}}q^{1/2},
\end{align*}
 the above estimates imply 
\begin{align*}
\frac{W^{2r}}{|\cU|^{2r}V^{2r}}\le 2^{12r+11}r^2(2r)^{\omega(q)}\tau(q)^{2r}\left(\frac{q}{\phi(q)}\right)^2N^{2r-2}q^{r/2(r-1)}(\log{q}),
\end{align*}
and hence 
\begin{align*}
\frac{W}{|\cU|V}\le \Delta_r N^{1-1/r}q^{1/4(r-1)},
\end{align*}
where $\Delta_r$ is given by~\eqref{eq:thm1deltadef}. Combining the above with~\eqref{eq:SSS} we get 
\begin{align*}
S\le 2\Delta_rq^{1/4(r-1)}N^{1-1/r},
\end{align*}
and completes the proof.
\section{Proof of Corollary~\ref{cor:main1}}
Assuming
$$N\ge q^{1/4+1/4\ell},$$
we apply Theorem~\ref{thm:main1} with 
$$r=2\ell+1$$
to get
\begin{align*}
\left|\sum_{M<n \le M+N}\chi(n)e_q(an)\right|\le 2\Delta_{2\ell+1}\frac{N}{q^{1/16\ell^2+8\ell}}.
\end{align*}
It remains to simplify the factor $\Delta_{2\ell+1}$. By Lemma~\ref{lem:omegaub} and Lemma~\ref{lem:phiUN} and Lemma \ref{lem:tauub} and remembering that $\log q \ge e^{16\ell^2}$, we  have
\begin{align*}
\Delta_{2\ell+1}\le q^{1.4/\log\log{q}}2^{7}(\log{q})^{1/4\ell}(\log\log{q})^{1/4\ell},
\end{align*}
which completes the proof.

\section{Proof of Theorem~\ref{thm:main}}
Assuming $\chi$ is primitive, we may expand into Gauss sums to get
\begin{align*}
\chi(n)=\frac{1}{\tau (\overline{\chi})}\sum_{a=1}^q \overline{\chi}(a) e\left( \frac{an}{q} \right) = \frac{1}{\tau (\overline{\chi})} \sum_{0< |a|< q/2} \overline{\chi}(a) e\left( \frac{an}{q} \right),
\end{align*}
 which after summing over $ 1 \leq n \leq N$ results in
\begin{align*}
\sum_{n=1}^N \chi(n)= \frac{1}{\tau (\overline{\chi})} \sum_{0< |a|< q/2} \overline{\chi}(a) \sum_{n=1}^N e\left( \frac{an}{q} \right)= \frac{1}{\tau (\overline{\chi})} \sum_{0< |a|< q/2} \overline{\chi}(a) \frac{ e\left( \frac{aN}{q} \right)-1}{1- e\left( \frac{-a}{q} \right)}.
\end{align*}
Since $|\tau (\overline{\chi})|=\sqrt{q}$  and 
\begin{align*}
\frac{ 1}{1- e\left( \frac{-a}{q} \right)} = \frac{q}{2 \pi i a} +\frac{\sum_{j=2}^{\infty} \frac{(-\frac{2 \pi i a}{q})^{j-2}}{j!} }{-\frac{q}{2\pi i a}\left( e(\frac{-a}{q})-1\right) },
\end{align*}
for $0< |a|< q/2$, it follows that
\begin{align*}
\left|\sum_{n=1}^N \chi(n)\right|&\le \frac{\sqrt{q}}{2 \pi} \left| \sum_{0< |a|< q/2}\frac{\overline{\chi(n)}\left( e(\frac{aN}{q})-1\right)}{a}\right| + \frac{5e^{\pi}}{2 \pi^2}\sqrt{q} \\
&\le \frac{\sqrt{q}}{2\pi}\left(\Sigma_1+\Sigma_2+\frac{5e^{\pi}}{2\pi^2}\right),
\end{align*}
where 
\begin{align*}
\Sigma_1=\sum_{0< |a|< q_1}\frac{\overline{\chi(n)}\left( e(\frac{aN}{q})-1\right)}{a},
\end{align*}
\begin{align*}
\Sigma_2= \sum_{q_1< |a|< q/2}\frac{\overline{\chi(n)}\left( e(\frac{aN}{q})-1\right)}{a},
\end{align*}
and
$$q_1=q^{\frac{1}{4}+\frac{1}{4\ell}}.$$
Assuming $$\log{q}\ge e^{\frac{(\ell^2+\ell/2) \alpha}{(\alpha -16)\ell^2-8\ell} 22.4 \ell^2},$$ by partial summation and Theorem \ref{cor:main2}, we have
\begin{align*}
|\Sigma_2| \le 2\log q\max_{q_1 \leq x \leq q} \left|\frac{1}{x} \sum_{a \leq x} \overline{\chi(a)}\left(e_q(aN)-1 \right)\right| \\ \le 4\log q ~\frac{2^{7}(\log{q})^{1/4\ell}(\log\log{q})^{1/4\ell}}{q^{1/\alpha \ell^2}}.
\end{align*}
Noting that
\begin{align*}
 \Sigma_1= 
\begin{cases}
     2i\sum\limits_{1\le a \le q_1}\frac{\overline{\chi(a)}\sin(\frac{2\pi aN}{q})}{a} ~~\text{if}~~\chi(-1)=1\\\\
    -2\sum\limits_{1\le a \le q_1}\frac{\overline{\chi(a)}\left( 1-\cos(\frac{2\pi aN}{q}\right)}{a}~~\text{if}~~\chi(-1)=-1,
\end{cases}
\end{align*}
by Lemma \ref{LI2}
\begin{align*}
  \left|\Sigma_1 \right| \le 
  \begin{cases}
     2 \left( \frac{2}{\pi} \log q_1 +\frac{2}{\pi}\left( \gamma +\log 2 +\frac{3}{q_1 }\right)\right)~~\text{if}~~\chi(-1)=1\\
    ~\\
    2\left(\log q_1 +\gamma +\log 2 +\frac{3}{q_1}\right)~~\text{if}~~\chi(-1)=-1.
  \end{cases}
\end{align*}
Thus, remembering the lower bound for $q$, we obtain the desired result
\begin{align*}
 & \left|\sum_{1\le n \le N}\chi(n)
   \right| \le \\ & 
  \begin{cases}
     \frac{2}{\pi^2}(\frac{1}{4}+\frac{1}{4\ell}) \sqrt{q}\log q +\left(6.5+\frac{2^{9}(\log{q})^{1/4\ell}(\log\log{q})^{1/4\ell}\log q}{\pi q^{ 1/\alpha \ell^2}}\right)\sqrt{q} ~~\text{if}~~\chi(-1)=1\\
    ~\\
   \frac{1}{\pi} (\frac{1}{4}+\frac{1}{4\ell}) \sqrt{q} \log q+ \left(6.5+\frac{2^{9}(\log{q})^{1/4\ell}(\log\log{q})^{1/4\ell}\log q}{\pi q^{1/\alpha \ell^2}}\right)\sqrt{q} ~~\text{if}~~\chi(-1)=-1.
  \end{cases}
\end{align*}

\section{Explicit upper bound for exceptional zeroes of Dirichlet $L$-functions.}
\label{sec:siegel}
In this section we obtain a new explicit bound for exceptional zeros of Dirichlet $L$-functions to real characters and will be based on combining Theorem~\ref{thm:main} with the argument of the first author~\cite{Bordignon}. Let $\chi$ a non-principal character mod $q$ and suppose the $L$-function
\begin{align*}
L(s,\chi):= \sum_{n=0}^{\infty} \chi(n)n^{-s},
\end{align*} 
has an exceptional zero $\beta_0$ satisfying
$$L(\beta_0,\chi)=0, \quad 1-\beta_0\ll \frac{1}{\log{q}},$$
Using computations from~\cite{McCurley, Trudgian, Platt}, we can focus on real zeros of non-principal real characters with modulus $q \geq 4\cdot 10^5$. Previous explicit estimates for $\beta_0$ have been of the form 
$$\beta_0 \leq 1 -\frac{\lambda}{  q^{1/2}\log^2 q},$$ 
for certain  $\lambda$. Some previous results are:
\begin{enumerate}
\item Liu and Wang prove $\lambda \approx 6$ for $q> 987$ in \cite[Theorem 3]{Liu-Wang},
\item Ford, Luca and Moree\ prove $\lambda \approx 19 $ for $q> 10^4$ in \cite[Lemma 3]{Ford},
\item Bennett, Martin O'Bryant and Rechnitzer  prove $\lambda = 40$ for $q> 4\cdot 10^5$ in \cite[Proposition 1.10]{Bennett},
\item Bordignon proves $\lambda = 80$ for $q> 4\cdot 10^5$ in \cite[Theorem 1.3]{Bordignon}.
\item Bordignon proves $\lambda = 100$ for $q> 4\cdot 10^5$ in \cite[Theorem 1.3]{Bordignon1}.
\end{enumerate}
We will prove the following result.
\begin{theorem}
\label{sz1}
For $\log q\ge e^{30\ell^2} \cdot $ and $\ell \ge 2$
\begin{align*}
\beta_0 \leq 1-\frac{3200 \pi\cdot (\frac{l}{l+1})^2}{\sqrt{q}\log^2 q}.
\end{align*}
\end{theorem}
It is enough to prove this result for primitive real characters. Indeed, if $\chi \pmod q$ is induced by a primitive real character $\chi' \pmod q'$, then the primitive case yields
\begin{align*}
\beta_0 \leq 1-\frac{\lambda}{\sqrt{q'}\log^2 q'}\leq 1-\frac{\lambda}{\sqrt{q}\log^2 q},
\end{align*}
with $\lambda$ as in Theorem~\ref{sz1}.
\newline

Using the mean value theorem it is easy to see that
\begin{equation}
\label{mvt}
1-\beta_0=\frac{L(1,\chi)}{| L'(\sigma,\chi)|},
\end{equation}
for some $\sigma \in (\beta_0,1)$. Thus we are left to obtain a lower bound for $L(1,\chi)$ and an upper bound for $| L'(\sigma,\chi) |$ for $\sigma \in (\beta_0,1)$.
\subsection{Lower bound for $L(1,\chi)$}
The following result is a consequence of the Class Number Formula and computations by Watkins~\cite{Watkins}.
\begin{lemma}
\label{le1}
For $q \ge e^{120}$, we have
\begin{align*}
L(1,\chi) \geq \frac{100\pi}{\sqrt{q}}.
\end{align*}
\end{lemma}
We use that every real primitive character can be expressed using the Kroneker symbol, as $\chi(n)= (\frac{d}{n})$, with $q= \left| d\right| $. From the Class Number Formula, for $d < 0$
\begin{align*}
L(1,\chi)=\frac{2 \pi h(\sqrt{d})}{\omega_d \sqrt{|d|}},
\end{align*}
where $h(\sqrt{d})$ is the class number of $\mathbb{Q}(\sqrt{d})$ and $\omega_d$ is the number of roots of unity of $\mathbb{Q}(\sqrt{d})$, that is known to be equal $2$ for $d<-3$. Now from~\cite[Table. 4]{Watkins} we have that, for $q \ge e^{120\pi}$, $h(\sqrt{d}) \ge 100$ and thus
\begin{align*}
L(1,\chi)\ge \frac{100 \pi }{\sqrt{|d|}}.
\end{align*}
If $d > 0$ then we have
\begin{align*}
L(1,\chi)=\frac{ h(\sqrt{d})\log \eta_d}{ \sqrt{|d|}},
\end{align*}
with $\eta_d=(v_0+u_0\sqrt{d})/2$ and $v_0$ and $u_0$ are the minimal positive integers satisfying $v_0^2-du_0^2=4$. From $h(\sqrt{d}) \ge 1$ and $$\eta_d \ge (\sqrt{d+4}+\sqrt{d})/2 \ge \sqrt{d},$$ we have, for $q \ge e^{120\pi}$,
\begin{align*}
L(1,\chi)\ge \frac{\log d}{2 \sqrt{d}}\ge \frac{100 \pi }{\sqrt{|d|}}.
\end{align*}
Thus Lemma \ref{le1} follows.
\subsection{Upper bound for $| L'(\sigma,\chi) |$ for $\sigma \in (\beta_0,1)$}
We are missing the following result to prove Theorem \ref{sz1}.
\begin{theorem}
\label{42}
Let $\chi$ be a primitive real character mod $q$ with $q$ satisfying $\log q> e^{30\ell^2}$. If $L(\sigma,\chi)$ has an exceptional zero $\beta_0$ satisfying $$\beta_0 \geqslant 1-\frac{3200\pi}{ \sqrt{p}\log^2 p},$$ then for $\sigma \in (\beta_0, 1)$ we have
\begin{align*}
\left| L'(\sigma, \chi) \right| \leq  \frac{1}{32} \left( \frac{l+1}{l}\right)^2\log^2 q.
\end{align*}
\end{theorem}
The following result is~\cite[Lemma~2.2]{Bordignon}.
\begin{lemma}
\label{11}
Let $g(n)$ be such that for all $n$ we have $g(n)\in \{-1, 0, 1\}$. 
We further assume that there is a $ M_g \in \R$ such~that
\begin{align*}
 \max_k \left|\sum_{n=0}^k g(n)\right| \leq M_g,
\end{align*}
and a $V(N)$, such that
\begin{align}
\label{20}
 \left|\sum_{n=0}^{N} g(n)\right| \leq V(N).
\end{align}
Let $f: \R  \rightarrow \R $ be such that $f \geq 0$, $f \rightarrow 0$, $f\in \mathcal{C'}$, $f'(x) <0$ and $\left| f'\right|  \searrow$ such~that
\begin{align*}
\int_0^{\infty} \left| f'(x) \right|  dx < \infty .
\end{align*}
If $V(N)\leqslant \min \{N, M_g\}$ holds true when
\begin{align*}
C_1\leqslant N\leqslant C_2,
\end{align*}
with $C_1, C_2 \in \mathbb{N}$, then $\left|\sum_{n=0}^{\infty} g(n)f(n)\right|$ has as an upper bound 
\begin{align}
\label{100}
\sum_{n=0}^{ C_1} f(n)+M(q) f(C_2)-C_1 f(C_1)-\int_{C_1}^{C_2} V(x) f'(x) dx .
\end{align}
\end{lemma}
We now prove Theorem \ref{42} applying Lemma \ref{11} to a real primitive character $\chi$. We refer the reader to~\cite{Bordignon} for more complete details. The bound (\ref{20}) will be the one in Corollary \ref{cor:main1}, with $a=0$ and $\log q> e^{30\ell^2}$.
We then have, using P\'{o}lya--Vinogradov, that $V(N) \leq \min \{ N,  q^{\frac{1}{2}} \log q\}$, when
\begin{align*}
q^{\frac{1}{4}+\frac{1}{4l}} \leq N \leq \frac{q^{1/(16\ell^2+8\ell)-1.4/\log\log{q}}}{2^{7}(\log{q})^{1/4\ell}(\log\log{q})^{1/4\ell}}q^{\frac{1}{2}} \log q.
\end{align*}

Note that it is possible to improve the result using the explicit P\'{o}lya--Vinogradov inequality from \cite{FS} or Corollary \ref{cor:main1}. This would only lead to minor improvements and is compensated by the fact that the chosen $f$ is decreasing.
Now using the explicit bound from~\eqref{100} with $C_1(q)= q^{\frac{1}{4}+\frac{1}{4l}}$ and $C_2(q)=C(q)q^{\frac{1}{2}} \log q$ with $C(q)=\frac{q^{1/(16\ell^2+8\ell)-1.4/\log\log{q}}}{2^{7}(\log{q})^{1/4\ell}(\log\log{q})^{1/4\ell}}$, and $\log q \ge e^{30 \ell^2}$, we obtain the following upper bound for $L'(\sigma,\chi)$
\begin{align*}
\sum_{n=2}^{ q^{\frac{1}{4}+\frac{1}{4l}}} \frac{\log n}{n^{\sigma}} +q^{\frac{1}{2}} \log q f(C_2(q))-q^{\frac{1}{4}+\frac{1}{4l}} \left(\frac{1}{4}+\frac{1}{4 \ell}\right)\frac{\log q}{q^{\sigma \left(\frac{1}{4}+\frac{1}{4l}\right)}}+\\-\int_{C_1(q)}^{C_2(q)} V(x) \frac{(1-\sigma)\log x}{x^{1+\sigma}} dx \le 
\end{align*}
\begin{align*}
\le q^{\frac{1-\beta_0}{2}}\frac{1}{32}\left(\frac{\ell+1}{l} \right)^2\log^2 q -\frac{1}{4}\log q+(C_2(q))^{1-\beta_0}C(q)^{-1}\log^2 C_2(q).
\end{align*}
Now, using that for $|x|\le 1$ we have $e^x \le 1+2x$ and remembering that $\sigma \in (\beta_0, 1)$ with $\beta_0 \geqslant~1-~\frac{3200\pi}{ \sqrt{p}\log^2 p}$ and that $\log q \ge e^{30 \ell^2}$, we obtain
\begin{align*}
\le \frac{1}{32}\left(\frac{\ell+1}{l} \right)^2\log^2 q +\frac{3200\pi \log q}{\sqrt{q}} -\frac{1}{4}\log q+e^{(1-\beta_0)\log(q)} \frac{\log^3 q}{q^{\frac{1}{35\ell^2}}} \le 
\end{align*}
\begin{align*}
\le \frac{1}{32}\left(\frac{\ell+1}{l} \right)^2\log^2 q .
\end{align*}
Thus Theorem \ref{42} follows. From this last result, Lemma~\ref{le1} and~\eqref{mvt} easily follows from Theorem \ref{sz1}.

\end{document}